\documentclass{amsart}
\usepackage{amsmath,amssymb}
\usepackage{graphicx}
\newtheorem{theorem}{Theorem}[section]
\newtheorem{proposition}[theorem]{Proposition}

\newtheorem{lemma}[theorem]{Lemma}

\theoremstyle{definition}

\newtheorem*{question}{Question}

\theoremstyle{remark}
\newtheorem{remark}[theorem]{Remark}

\numberwithin{equation}{section}

\newcommand{\de}{\delta}
\newcommand{\ep}{\epsilon}

\newcommand{\ga}{\gamma}

\newcommand{\la}{\lambda}

\newcommand{\te}{\theta}
\newcommand{\vp}{\varphi}

\newcommand{\De}{\Delta}

\newcommand{\La}{\Lambda}
\newcommand{\Si}{\Sigma}
\newcommand{\Om}{\Omega}
\newcommand{\Omc}{{\Omega^c}}

\def\RR{\mathbb{R}}

\def\ZZ{\mathbb{Z}}

\renewcommand\SS{\mathbb{S}}
\newcommand{\cW}{{\mathcal W}}

\newcommand{\pd}{\partial}
\newcommand\minus\backslash

\newcommand\lan\langle
\newcommand\ran\rangle

\newcommand{\rank}{\operatorname{rank}}

\newcommand{\Span}{\operatorname{span}}

\newcommand\DD{\mathbb D}
\renewcommand\leq\leqslant
\renewcommand\geq\geqslant
\newlength{\intwidth}

\begin{document}

\title[Laplace operators with eigenfunctions whose nodal set is a knot]{Laplace operators with eigenfunctions\\ whose nodal set is a knot}

\author{Alberto Enciso}
\address{Instituto de Ciencias Matem\'aticas, Consejo Superior de
  Investigaciones Cient\'\i ficas, 28049 Madrid, Spain}
\email{aenciso@icmat.es, david.hartley@icmat.es, dperalta@icmat.es}

\author{David Hartley}
%\address{Instituto de Ciencias Matem\'aticas, Consejo Superior de
%  Investigaciones Cient\'\i ficas, 28049 Madrid, Spain}
%\email{david.hartley@icmat.es}

\author{Daniel Peralta-Salas}
%\address{Instituto de Ciencias Matem\'aticas, Consejo Superior de
%  Investigaciones Cient\'\i ficas, 28049 Madrid, Spain}
%\email{dperalta@icmat.es}

%%    General info
%\subjclass[2010]{35B38, 58J05, 58K45}
%\date{\today}
%
%\keywords{ }
%
\begin{abstract}

We prove that, given any knot $\ga$ in a compact $3$-manifold
$M$, there exists a Riemannian metric on $M$ such that there is a
complex-valued eigenfunction $u$ of the Laplacian, corresponding to the first
nontrivial eigenvalue, whose nodal set $u^{-1}(0)$ has a
connected component given by $\ga$. Higher dimensional analogs of this
result will also be considered.
\end{abstract}
\maketitle

\section{Introduction}
\label{S.intro}

Let $M$ be a closed manifold of dimension
$3$ endowed with a smooth Riemannian metric $g$. The
eigenfunctions of $M$ satisfy the equation
\[
\De u_k=-\la_k u_k\,,
\]
where $0=\la_0<\la_1\leq\la_2\leq\dots$
are the eigenvalues of $M$ and $\De$ is the Laplace operator of the
manifold. As is well known, the zero set $u_k^{-1}(0)$ is called the {\em nodal set}\/
of the eigenfunction. 

Our goal in this paper is to establish the existence of knotted
nodal sets in eigenfunctions of the Laplacian on a Riemannian
3-manifold. Specifically, the starting point of the present paper is
the following

\begin{question}
Let $\ga$ be a (possibly knotted) curve in $M$. Is there a complex-valued eigenfunction~$u$ of the Laplacian whose
nodal set $u^{-1}(0)$ is diffeomorphic to $\ga$?
\end{question}

Before getting into the background of this question, it is worth
discussing why the eigenfunction $u$ is assumed complex-valued. The
idea is that, by the maximum principle, the nodal set of a real-valued eigenfunction
(i.e., $u_k^{-1}(0)$) cannot have a connected component of codimension
greater than or equal to 2, so there cannot be a real eigenfunction of
the Laplacian whose nodal set has a component diffeomorphic to the curve~$\ga$. A way to bypass this obstruction is
to take a complex function of the form, say, $u:=u_1+i u_2$, which will
indeed be an eigenfunction provided that $\la_1=\la_2$. Notice that $u^{-1}(0)=u_1^{-1}(0)\cap u_2^{-1}(0)$. Hence the
above question is in fact a problem about Laplacians with degenerate
eigenvalues.

The context of this question is the study of the nodal sets of the
eigenfunctions of the Laplacian in a compact Riemannian manifold,
which is a classical topic in geometric analysis with a number of
important open problems~\cite{Ya82,Ya93}. In a way, this question
addresses the flexibility of low-energy Laplace eigenfunctions, as we
shall explain next.

It is well known that high-energy eigenvalues and
eigenfunctions of the Laplacian are rather rigid, as illustrated by
Weyl's asymptotic formula for the eigenvalues and by Yau's conjecture
concerning the Hausdorff measure of the eigenfunction's nodal sets as
the eigenvalue tends to infinity, which was proved for real-analytic
metrics in~\cite{DF}. Indeed, the measure-theoretic properties of the
nodal set of high-energy eigenfunctions have been a hot topic of
research for decades~\cite{DF,HS,Lin}; a recent account of the subject
can be found in~\cite{JNT}. 

On the contrary, the situation for low-energy eigenvalues and eigenfunctions is much
more flexible. The eigenvalue case was considered by Colin de Verdi\`ere
in~\cite{Colin2}, where it was shown that there are metrics
whose first $N$ eigenvalues can be prescribed a priori. Concerning
nodal sets, the problem of constructing a metric on $M$ for which the nodal set $u_1^{-1}(0)$ of the first
eigenfunction has a prescribed connected component was solved
in~\cite{EPpre} (see also~\cite{EJN,Komen} for the two dimensional
case). 

The question that we have started with addresses the possibility
of prescribing a component of the intersection of nodal sets
\[
u^{-1}(0)=u_1^{-1}(0)\cap u_2^{-1}(0)\,,
\]
subject to the additional constraint that the eigenvalues $\la_1$ and
$\la_2$ should coincide. This constraint is important, since it is
well known that for a generic metric the eigenvalues are all
simple~\cite{Uh76}. In a way, this problem is the metric counterpart
of Berry's conjecture~\cite{Be01}, recently proved in~\cite{EHP},
which states that there should be a smooth (real) potential for which
the Schr\"odinger operator $-\Delta+V$ in $\RR^3$ admits a
complex-valued eigenfunction whose nodal set has a connected component
of any given knot type.

Our main result in this paper provides an affirmative answer to the
above question. Specifically, we show that given a smooth closed
curve~$\ga$ (without self-intersections, but possibly knotted), there
is a metric for which $u^{-1}(0)$ has a connected component given by
$\ga$, with $u:=u_1+i u_2$ corresponding to the first eigenvalue
$\la_1=\la_2$, which has multiplicity 2. Moreover, the knot $\ga$ is
{\em structurally stable}\/ in the sense that any small enough perturbation of
$u$ (in the $C^k$ norm with $k\geq1$) has a connected component in its
nodal set that is a small perturbation of~$\ga$ (i.e., the image of
$\ga$ under a smooth diffeomorphism of $M$ that is close to the
identity in the $C^k$ norm). We are thus led to the following:

\begin{theorem}\label{T.main}
Let $\ga$ be a (possibly knotted) closed curve in a closed $3$-manifold $M$. Then there
exists a Riemannian metric $g$ on~$M$ such that its first nontrivial
eigenvalue has multiplicity $2$ (i.e., $\la_1=\la_2\neq\la_3$) and $\ga$ is a
connected component of the nodal set $u^{-1}(0)$ of the complex-valued
eigenfunction $u=u_1+i u_2$. Moreover, $\ga$ is a structurally stable
component of the nodal set. 
\end{theorem}

Equivalently, the theorem asserts that the closed curve~$\ga$ is a
connected component of the intersection of the nodal sets of the
real-valued eigenfunctions $u_1^{-1}(0)\cap u_2^{-1}(0)$ for any basis
of the eigenspace corresponding to the first nontrivial eigenvalue.

The proof of this result starts off with the construction of a metric
$g$ that is of order~$\ep$ in the complement of a tubular neighborhood
$\Om$ of the curve $\ga$. If $\ep$ is small enough, the eigenfunctions
$u_1$ and $u_2$ are close in the set~$\Om$ to the first nontrivial Neumann
eigenfunctions $v_1$ and $v_2$ of this set, so in principle the nodal set of the former can be controlled in terms of the latter using Thom's isotopy theorem. To develop this strategy there are two difficulties one has to deal with. First, the intersection $v_1^{-1}(0)\cap v_2^{-1}(0)$ must be transverse in order to apply Thom's isotopy theorem. Second, the eigenvalue $\la_1$ must be degenerate, which is a non-generic property according to Uhlenbeck's theorem~\cite{Uh76}. These problems are overcome by constructing very carefully the metric $g$ in successive steps.

Since the nodal set of the complex-valued eigenfunction $u=u_1+i u_2$
is the intersection of the nodal sets $u_1^{-1}(0)\cap u_2^{-1}(0)$,
the question that we have started our paper with is simply the easiest
case of whether we can realize a set of codimension $m\geq2$ in a
compact $d$-manifold as the intersection of the nodal sets
\[
u_1^{-1}(0)\cap \cdots \cap u_m^{-1}(0)
\]
of $m$
real-valued eigenfunctions having the same energy $\la_1=\cdots= \la_m$ (of
multiplicity $m$). The
full problem will be tackled in Section~\ref{S.higher}, where we
provide a higher-dimensional version of Theorem~\ref{T.main} (Theorem~\ref{T.high}). This is
a significant generalization of~\cite{EPpre}, but its statement is a
little more technical than that of Theorem~\ref{T.main}.

The content of this paper is as follows. The proof of Theorem~\ref{T.main} is presented in Section~\ref{S.main}, although the proofs of two technical results, Propositions~\ref{lamconv} and~\ref{evalcol2}, are relegated to
Sections~\ref{S:prop} and~\ref{S.lemma}, respectively. Finally, in
Section~\ref{S.higher} we state and prove a higher dimensional analog of Theorem~\ref{T.main}.

\section{Proof of the main theorem}
\label{S.main}

The proof of Theorem~\ref{T.main} is divided in four steps. In Step~1
we introduce an appropriate metric $g_0$ whose explicit expression in
a tubular neighborhood $\Om$ of the curve $\ga$ allows us to compute,
in closed form, the first Neumann eigenvalues and eigenfunctions of the
domain $\Om$. In Step~2 we rescale the metric $g_0$ with a smooth
factor that is of order $\ep$ in the complement of $\Om$, and show
that the eigenvalues and eigenfunctions of $M$ with this metric tend
to the Neumann eigenvalues and eigenfunctions of $\Om$ as $\ep\searrow
0$ (Proposition~\ref{lamconv}). In Step~3 we prove that the metric
introduced in the previous step can be deformed so that its first
nontrivial eigenvalue has multiplicity 2 and its corresponding
eigenspace is close to an eigenspace of Neumann eigenfunctions of
$\Om$ (Proposition~\ref{evalcol2}). Finally, in Step~4 we use Thom's
isotopy theorem to control the intersection of the nodal sets of the
eigenfunctions obtained before.        

\subsubsection*{Step 1: Neumann
  eigenfunctions of a certain bounded domain}

Let us consider a tubular neighborhood $\Om$ of the curve $\ga$,
which is diffeomorphic to the product $\SS^1\times\DD^2$ because the
normal bundle of a knot is always trivial~\cite{Ma59}. Hence we will
parametrize the set $\Om$ using coordinates $z=(z_1,z_2)\in\DD^2$  and
$s\in\RR/\ZZ$. Here $\DD^2$ denotes the unit
two-dimensional disk. 

We shall take a metric $g_0$ on $M$ such that its restriction to $\Om$
can be written in these coordinates as
\[
g_0|_\Om:=A^{-1}\, ds^2+|dz|^2=A^{-1}\, ds^2+ dr^2+r^2\, d\te^2\,.
\]
Here $(r,\te)$ are the polar coordinates in the disk associated to $z$
and $A$ is a large enough constant. Specifically, we assume that 
\begin{equation}\label{eqA}
A>\frac{\mu_{\DD^2}}{4\pi^2}\,,
\end{equation}
where $\mu_{\DD^2}$ is the first nontrivial Neumann eigenvalue of the unit
$2$-disk. Notice that the coordinate $s$ essentially corresponds to a
parametrization of the curve~$\ga$ proportional to arc-length and that
the variables~$z$ describe a disk orthogonal to the curve.

Let us consider the Neumann eigenvalue problem
\[
\De v_k=-\mu_k v_k\quad \text{in } \Om\,,\qquad \pd_\nu v_k=0\quad \text{on } \partial\Om\,,
\]
where $\De$ stands for the Laplacian associated with the metric~$g_0$.
By separation of variables and using the condition~\eqref{eqA}, we infer that the Neumann eigenvalues $\mu_k$ of the domain
$\Om$ with respect to the metric~$g_0$ are then $\mu_0=0$,
\[
\mu_1=\mu_2=\mu_{\DD^2}\,,
\]
and $\mu_k>\mu_{\DD^2}$ for $k\geq3$. Hence the eigenspace of $\mu_1$ is two-dimensional, with a particular orthogonal basis being
\[
v_1=J_1(\sqrt{\mu_1}r)\, \cos\te\,,\qquad v_2= J_1(\sqrt{\mu_1} r)\, \sin\te\,.
\]
Here $J_1$ denotes the Bessel function.

Since $J_1(\sqrt{\mu_1}r)$ does not vanish for $0<r<1$ and the eigenfunctions
behave in a neighborhood of the origin as 
\[
v_1=\frac{\sqrt{\mu_1}}2 r\,\cos\te+ O(r^2)=\frac{\sqrt{\mu_1}}2 z_1+ O(|z|^2)\,,\qquad v_2=\frac{\sqrt{\mu_1}}2 z_2+ O(|z|^2)\,,
\]
it is elementary that, as the eigenfunctions are only defined in the
set $|z|<1$, the intersection
\[
v_1^{-1}(0) \cap v_2^{-1}(0)
\]
is the set $z=0$, that is, the curve $\ga$ (which is identified via
the coordinates with the set $\SS^1\times\{0\}$). Moreover, the
behavior of the Neumann eigenfunctions near the set $z=0$ shows that their
intersection is {\em transverse}\/ in the sense that
\begin{equation}\label{rank}
\rank(\nabla v_1(x),\nabla v_2(x))=2
\end{equation}
for all $x\in v_1^{-1}(0) \cap v_2^{-1}(0)$. These properties obviously
hold for any other basis $\{v_1,v_2\}$ of the eigenspace
corresponding to $\mu_1$.

\subsubsection*{Step 2: Convergence of eigenvalues and eigenfunctions}

Let us take a one-parameter family of smooth, positive, uniformly bounded functions $f_t$ on
$M$ such that $f_t(x)=\ep$ in the complement $\Om^c:=M\backslash\Om$
and $f_t(x)=1$ if $x\in\Om$ and the distance between $x$ and $\pd\Om$,
as measured with
the metric $g_0$, is larger than $1/t$. Here both $t$ and $\ep$ are positive real numbers, and $f_t$ is assumed to be analytic in $t$ and $\ep$. Notice that the smooth
functions $f_t$ converge pointwise to the discontinuous function
\[
f(x):=\begin{cases}
\ep& \text{if }x\in \Om^c\,,\\
1& \text{if }x\in \Om\,,
\end{cases}
\]
as $t\to \infty$.

Let us now define the one-parameter family of smooth metrics $g_t:=f_t\,g_0$. In the
following proposition we consider the behavior of the eigenvalues
$\la_{k,t}$ and eigenfunctions~$u_{k,t}$ of the metric $g_t$ as the
parameter~$\ep$ tends to zero and $t$ tends to infinity. Specifically,
we are interested in comparing these quantities with the Neumann
eigenvalues~$\mu_k$ and eigenfunctions~$v_k$ of the domain $\Om$,
which were introduced in Step~1:

\begin{proposition}\label{lamconv}
There is a basis of the eigenfunctions $u_{k,t}$ of the Laplacian on $M$ with
the metric~$g_t$ and a basis $v_k$ of the Neumann eigenfunctions of the
domain~$\Om$ such that
\begin{align*}
\la_{k,t}-\mu_k&=o(1)\,,\\
\|u_{k,t}-v_k\|_{H^1(\Om)}&=o(1)\,.
\end{align*}
Here we are denoting by $o(1)$ quantities that tend to zero (not
uniformly in $k$) as $\ep\searrow0$ and $t\to\infty$.
\end{proposition}

The proof of this result is presented in Section~\ref{S:prop}, and uses ideas
introduced by Colin de Verdi\`ere~\cite{Colin}.

\subsubsection*{Step 3: Collapse of eigenvalues}
From Proposition~\ref{lamconv} we infer that the eigenvalues
$\la_{1,t}$ and $\la_{2,t}$ are very close to the first Neumann
eigenvalue $\mu_{\DD^2}$ provided that $t$ is large enough and $\ep$
is sufficiently small. Uhlenbeck's theorem~\cite{Uh76} implies that,
generically, $\la_{1,t}\neq \la_{2,t}$, so generally if we want to get a sequence of eigenvalues of multiplicity~2 converging to $\mu_{\DD^2}$, we need to deform the family of metrics $g_t$. In the following proposition we show that such a deformation exists. The statement of the proposition is in terms of an $\ep$-dependent smooth metric $\overline g_{\ep}$, $\ep>0$, whose eigenvalues will be denoted by~$\overline\la_{k,\ep}$.

\begin{proposition}\label{evalcol2}
There exists a positive $\ep$-dependent function $h_\ep$  on $M$ of class
$C^\infty$, a basis of the eigenfunctions $\overline u_{k,\ep}$
of the Laplacian on $M$ with the metric $\overline g_{\ep}:=h_\ep g_{T}$, where $T:=\ep^{-1}$,
and a basis $v_{k}$ of the Neumann eigenfunctions of the domain $\Om$
(with the metric $g_0$) such that:
\begin{enumerate}
\item $\overline \la_{1,\ep}$ has multiplicity $2$ and $\overline \la_{1,\ep}=\overline \la_{2,\ep}=\mu_{\DD^2}$,
\item $\lim_{\ep\searrow 0}\|\overline u_{1,\ep}-v_{1}\|_{H^1(\Om)}=\lim_{\ep\searrow 0}\|\overline u_{2,\ep}-v_{2}\|_{H^1(\Om)}=0$.
\end{enumerate}
Moreover, $\lim_{\ep\searrow 0}\|h_\ep-1\|_{C^l(M)}=0$ for any integer $l$.
\end{proposition}

The proof of this result is in Section~\ref{S.lemma} and exploits the fact that when $\ep\searrow0$ the eigenvalues $\la_{1,T}$ and $\la_{2,T}$ are very close according to Proposition~\ref{lamconv} (recall that $\mu_1=\mu_2=\mu_{\DD^2}$). Then a Hadamard-type formula for the variation of eigenvalues with respect to metric perturbations allows us to use an implicit function theorem argument to construct the function $h_\ep$.

\subsubsection*{Step 4: Transverse intersection of nodal sets}
Proposition~\ref{evalcol2} ensures that, for any fixed $\de>0$, one
can then choose a small $\ep$ so that the eigenvalue $\overline
\la_{1,\ep}$ of the Laplacian computed with the metric $\overline
g_\ep$ has multiplicity $2$, and that there is a basis $\{\overline
u_{1,\ep},\overline u_{2,\ep}\}$ of the corresponding eigenspace such that
\begin{align}\label{eq:estd}
\|\overline u_{1,\ep}-v_1\|_{H^1(\Om)}<\delta\,,\qquad \|\overline u_{2,\ep}-v_2\|_{H^1(\Om)}<\delta\,.
\end{align}
Here $\{v_1,v_2\}$ is a basis of Neumann eigenfunctions of the domain $\Om$.

Since the conformal factor $h_\ep$ defining the metric $\overline g_\ep$ in Proposition~\ref{evalcol2} is arbitrarily close to $1$ as $\ep\searrow 0$, the definition of the metric $g_t$ implies that 
$$\|\overline g_\ep-g_0\|_{C^l(S)}<C_l\delta$$
for any integer $l$ and any compact subset $S\subset \Om$, where the constant $C_l$ depends on $l$ and $S$. Therefore, standard elliptic estimates allow us to promote the $H^1$ bound~\eqref{eq:estd} to a $C^1$ bound:
\begin{align}\label{eq:estc1}
\|\overline u_{1,\ep}-v_1\|_{C^1(S)}<C\delta\,,\qquad \|\overline u_{2,\ep}-v_2\|_{C^1(S)}<C\delta\,.
\end{align}

For small enough~$\de$ (and hence small $\ep$), the transversality
property~\eqref{rank} allows us to apply Thom's isotopy
theorem~\cite[Section 20.2]{AR} to conclude that the nodal set
$\overline u_{1,\ep}^{-1}(0)\cap \overline u_{2,\ep}^{-1}(0)$ has a
connected component in $\Om$ diffeomorphic to 
\[
\ga=v_1^{-1}(0)\cap v_2^{-1}(0)\,,
\]
and that the corresponding diffeomorphism $\Phi_\ep:M\to M$ can be chosen arbitrarily close to the identity in $C^1(M)$ and different from the identity only in a small neighborhood of $\ga$.

Finally, let us consider the pulled-back metric
\[
g:=\Phi_\ep^*\overline g_{\ep}
\]
and call
\begin{align*}
u_1:=\overline u_{1,\ep}\circ \Phi_\ep\,,\qquad u_2:=\overline u_{2,\ep}\circ \Phi_\ep
\end{align*}
the associated eigenfunctions, whose corresponding eigenvalue $\la_1:=\overline \la_{1,\ep}$ has multiplicity $2$. Accordingly, the nodal set $u_1^{-1}(0)\cap u_2^{-1}(0)$ has a connected component given by the curve $\ga$, which is structurally stable because Eq.~\eqref{rank} and the $C^1$ estimates~\eqref{eq:estc1} imply
\begin{equation*}
\rank(\nabla u_1(x),\nabla u_2(x))=2
\end{equation*}
for all $x\in\ga$. The theorem then follows.

\begin{remark}\label{R:1}
The transversality property~\eqref{rank} of the intersection of the
nodal sets $v_1^{-1}(0)$ and $v_2^{-1}(0)$ is key in the proof, as in
particular it allows us to use Thom's isotopy theorem in Step~4.
\end{remark}

\section{Proof of Proposition~\ref{lamconv}}\label{S:prop}

Since the family of smooth functions $f_t$ converges pointwise to the discontinuous function $f$ as $t\to\infty$, it is convenient to work first with the discontinuous metric 
\begin{equation*}
g_{\epsilon}:=f g_0=\begin{cases}
\ep g_0& \text{if }x\in \Om^c\,,\\
g_0& \text{if }x\in \Om\,,
\end{cases}
\end{equation*}
and later we will prove the desired result for the metric $g_t$ provided that $t$ is large enough.

To define the spectrum of the Laplacian associated to the metric $g_\ep$, we use the quadratic form
\begin{equation*}
Q_{\epsilon}(\vp):=\int_M|d\vp|_{\epsilon}^2\,dV_{\epsilon}=\int_{\Omega}|d\vp|^2+
\epsilon^{\frac{1}{2}}\int_{\Omega^c}|d\vp|^2\,,
\end{equation*}
and the natural $L^2$ norm corresponding to the metric $g_\ep$:
\begin{align*}
\|\vp\|_\ep^2&:=\int_M \vp^2\, dV_\ep=\int_\Om \vp^2+ \ep^{\frac 32}\int_\Omc \vp^2\,.
\end{align*}
Here the subscripts~$\ep$ refer to quantities computed with respect to the
metric $g_\ep$ and we are omitting the subscripts (and indeed the
measure in the integrals) when the quantities correspond to the
reference metric $g_0$. As is well known, the domain of the quadratic
form $Q_\ep$ can be taken to be the Sobolev space $H^1(M)$. (Recall
that, $M$ being compact, this Sobolev space is independent of the
smooth metric one uses to define it).

By the min-max principle, the $k^{\mathrm{th}}$ eigenvalue
$\la_{k,\ep}$ of this quadratic form is
\begin{equation}\label{minmax}
\la_{k,\ep}=\inf_{W\in\cW_k} \max_{\vp\in W\minus\{0\}} q_\ep(\vp)\,,
\end{equation}
where $\cW_k$ stands for the set of $(k+1)$-dimensional linear subspaces
of $H^1(M)$ and
\begin{equation*}\label{qep}
q_\ep(\vp):=\frac{Q_\ep(\vp)}{\|\vp\|_\ep^2}
\end{equation*}
is the Rayleigh quotient. The $k^{\mathrm{th}}$ eigenfunction $u_{k,\ep}$ is
then a minimizer of the above variational problem for
$\la_{k,\ep}$, in the sense that any subspace that minimizes the variational
problem can be written as $\Span\{ u_{0,\ep},\dots, u_{k,\ep}\}$.

In what follows we fix an integer $k_0$. Let us first prove that the eigenvalues $\la_{k,\ep}$ are almost upper bounded by $\mu_k$ for $k\leq k_0$ in the sense that
\begin{equation}\label{eq:up}
\la_{k,\ep}\leq \mu_k+O(\ep^{\frac{1}{2}})\,,
\end{equation} 
where $O(\epsilon^{m})$ stands for a quantity that is bounded by
$C\ep^{m}$. For this purpose, we consider the harmonic extension to the whole manifold $M$ of the Neumann eigenfunctions $v_k$ of the domain $\Omega$:
\begin{equation*}
\psi_k:=\left\{\begin{array}{ll}\hat v_k & \text{in }{\Omega^c}\,,\\ {v}_k & \text{in }\Omega\,,\end{array}\right.
\end{equation*}
where $\hat{v}_k$ is the solution to the boundary value problem
\begin{equation*}
\Delta\hat{v}_k=0\text{ in }\Omega^c,\ \ \hat{v}_k|_{\partial\Omega}=v_k|_{\partial\Omega}\,.
\end{equation*}
Recall that $\Delta$ is the Laplacian computed with the reference metric $g_0$. Throughout this section we shall normalize the Neumann eigenfunctions $v_k$ so that they define an orthonormal basis of $L^2(\Om)$. Standard elliptic estimates then give us the bound
\begin{equation}\label{harmbound}
\int_{\Omega^c}|d\hat{v}_k^2|+\int_{\Omega^c}\hat{v}_k^2\leq C\|v_k\|^2_{H^{\frac{1}{2}}\left(\partial\Omega\right)}\leq C\|v_k\|_{H^1(\Omega)}^2=C(\mu_k+1)\,,
\end{equation}
where $C$ is a constant independent of $\epsilon$ and $k$. 

For any linear combination $\vp,\tilde{\vp}\in\text{span}\{\psi_0,\ldots,\psi_{k}\}$, with $k\leq k_0$, using Eq.~\eqref{harmbound} we obtain that
\begin{align*}
\int_M\vp\tilde{\vp}\,dV_{\epsilon}=& \int_{\Omega}\vp\tilde{\vp}+\epsilon^{\frac{3}{2}}\int_{\Omega^c}\vp\tilde{\vp}\nonumber\\
=&\int_{\Omega}\vp\tilde{\vp}+O(\epsilon^{\frac{3}{2}})(\mu_{k_0}+1)\|\vp\|_{L^2(\Omega)}\|\tilde{\vp}\|_{L^2(\Omega)},\nonumber\\
=&\int_{\Omega}\vp\tilde{\vp}+O(\epsilon^{\frac{3}{2}})\|\vp\|_{L^2(\Omega)}\|\tilde{\vp}\|_{L^2(\Omega)},
\end{align*}
where, to pass to the last line, we have absorbed the constant $\mu_{k_0}+1$ in the $O(\ep^{\frac{3}{2}})$ term. Accordingly, for sufficiently small $\epsilon$, the linear space
\begin{equation*}
\text{span}\{\psi_0,\ldots,\psi_k\},
\end{equation*}
is a $(k+1)$-dimensional subspace of $H^1(M)$ provided that $k\leq k_0$. In a similar manner we have
\begin{align*}
Q_{\epsilon}(\vp)=& \int_{\Omega}|d\vp|^2+ \epsilon^{\frac{1}{2}} \int_{\Omega^c}|d\vp|^2\nonumber\\
=&\int_{\Omega}|d\vp|^2+O(\epsilon^{\frac{1}{2}})\|\vp\|_{L^2(\Omega)}^2\,.
\end{align*}
Therefore, for $k\leq k_0$ we can estimate $\la_{k,\ep}$ as
\begin{align*}
\lambda_{k,\ep}\leq&\max_{\vp\in\{\psi_0,\ldots,\psi_k\}}\frac{Q_{\epsilon}(\vp)}{\|\vp\|_\ep^2}\\
\leq& \max_{\vp\in\{\psi_0,\ldots,\psi_k\}}\frac{\int_{\Omega}|d \vp|^2+O(\epsilon^{\frac{1}{2}})\|\vp\|_{L^2(\Omega)}^2}{\|\vp\|_{L^2(\Omega)}^2}\\
=&\max_{\chi\in\{v_0,\ldots,v_k\}}\frac{\int_{\Omega}| d\chi|^2}{\|\chi\|_{L^2(\Omega)}^2} +O(\epsilon^{\frac{1}{2}})\\
=&\mu_k +O(\epsilon^{\frac{1}{2}})\,,
\end{align*}
which proves the upper bound~\eqref{eq:up}.

Now we proceed to prove the lower bound
\begin{equation}\label{low}
\la_{k,\ep}\geq\left(1+O(\epsilon^{\frac{1}{2}})\right)\mu_k
\end{equation} 
for all $k\leq k_0$. To obtain this bound we split $H^1(M)$ into two closed linear subspaces
\begin{align*}
&\mathcal{H}_1:=\{\vp\in H^1(M):\Delta\vp=0\text{ in }\Omega^c\}\,,\ \\ &\mathcal{H}_2:=\{\vp\in H^1(M):\vp|_{\Omega^c}\in H^1_0(\Omega^c),\ \vp|_{\Omega}=0\}\,,
\end{align*}
so that $H^1(M)=\mathcal{H}_1\oplus\mathcal{H}_2$ and
\begin{equation*}
\int_M g_0(\nabla\vp_1,\nabla\vp_2)=0,
\end{equation*}
for all $\vp_j\in\mathcal{H}_j$. We will write $\vp_j$ for the component of a function $\vp\in H^1(M)$ in $\mathcal{H}_j$. 

We first note that by choosing a sufficiently small $\epsilon$, the upper bound~\eqref{eq:up} implies that $\lambda_{k,\ep}<\mu_{k_0}+1$ for all $k\leq k_0$, which allows us to reduce the min-max formula~\eqref{minmax} for $k\leq k_0$ to
\begin{equation}\label{redminmax}
\lambda_{k,\ep}=\inf_{W\in\mathcal{W}'_{k}}\max_{\vp\in W\backslash\{0\}}q_\ep(\varphi)\,,
\end{equation}
where $\mathcal{W}'_{k}$ is the set of $(k+1)$-dimensional subspaces of $H^1(M)$ such that
\begin{equation*}
q_\ep(\varphi)<\mu_{k_0}+1
\end{equation*}
for all nonzero $\vp\in W\subset \mathcal W'_k$.

A simple computation shows that
\begin{equation}\label{estp2}
\|\vp_2\|_{\epsilon}^2=O(\ep)\|\vp\|_{\epsilon}^2\,,
\end{equation}
for any $\vp\in W\subset \mathcal W'_k$ with $k\leq k_0$. Indeed, this relation follows from the estimate 
\begin{align*}
\mu_{k_0}+1>q_\ep(\vp)=\frac{\|\vp_2\|_\ep^2}{\|\vp\|^2_\ep}\frac{Q_\ep(\vp)}{\|\vp_2\|_\ep^2}\geq \frac{1}{\ep}\frac{\|\vp_2\|_\ep^2}{\|\vp\|^2_\ep}\frac{\int_{\Om^c}|d\vp_2|^2}{\int_{\Om^c}\vp_2^2}
\geq \frac{\la_{\Om^c}}{\ep}\frac{\|\vp_2\|_\ep^2}{\|\vp\|^2_\ep}\,,
\end{align*}
where the $\ep$-independent constant $\la_{\Om^c}$ is the first Dirichlet eigenvalue of $\Om^c$. 
In particular, Eq.~\eqref{estp2} implies that
\begin{equation}\label{eqc}
\|\vp\|_{\epsilon}^2=\left(1+O(\epsilon^{\frac{1}{2}})\right)\|\vp_1\|_{\epsilon}^2\,.
\end{equation} 

Proceeding as in the proof of Eq.~\eqref{harmbound} and using the bound~\eqref{eqc}, we obtain that for $\vp\in W\subset\mathcal W'_k$ with $k\leq k_0$,
\begin{equation*}
\int_{\Omega^c}\vp_1^2+\int_{\Omega^c}|d\vp_1|^2\leq C\int_{\Omega}\vp_1^2\,.
\end{equation*}

Combining these two results we conclude that, 
\begin{align*}
q_\ep(\vp)=&\left(1+O(\epsilon^{\frac{1}{2}})\right)\frac{\int_{M}|d \vp_1|_{\epsilon}^2\,dV_{\epsilon} +\int_{M}|d \vp_2|_{\epsilon}^2\,dV_{\epsilon}}{\|\vp_1\|_{\epsilon}^2}\\
\geq& \left(1+O(\epsilon^{\frac{1}{2}})\right)\frac{\int_{\Omega}|d \vp_1|^2}{\int_{\Omega}\vp_1^2 +\epsilon^{\frac{3}{2}}\int_{\Omega^c}\vp_1^2}\\
=& \left(1+O(\epsilon^{\frac{1}{2}})\right)\frac{\int_{\Omega}|d \vp_1|^2}{\int_{\Omega}\vp_1^2}\,,
\end{align*}
for any $\vp\in W\subset \mathcal W'_k$ with $k\leq k_0$. Therefore by the min-max principle~\eqref{redminmax} we have
\begin{align*}
\lambda_{k,\ep}\geq& \left(1+O(\epsilon^{\frac{1}{2}})\right)\inf_{W\in\mathcal W'_k}\max_{\vp\in W\backslash\{0\}}\frac{\int_{\Omega}|d \vp_1|^2}{\int_{\Omega}\vp_1^2}\\
\geq&\left(1+O(\epsilon^{\frac{1}{2}})\right)\mu_k,
\end{align*}
for all $k\leq k_0$, which establishes the lower bound~\eqref{low}.

The bounds~\eqref{eq:up} and~\eqref{low} imply that
\begin{equation}\label{conveig}
\lim_{\ep\searrow 0}\|\la_{k,\ep}-\mu_k\|=0
\end{equation}
for all $k\leq k_0$. Since the metric $g_\ep$ is equal to $g_0$ in the domain $\Om$, and the eigenvalues control the $H^1$ norm, it is standard that the convergence of eigenvalues~\eqref{conveig} implies that for $k\leq k_0$, there exists a basis of eigenfunctions $u_{k,\ep}$ and a basis of Neumann eigenfunctions $v_k$ such that
\begin{equation}\label{conveigenf}
\lim_{\ep\searrow 0}\|u_{k,\ep}-v_k\|_{H^1(\Om)}=0\,.
\end{equation}

To finish the proof of the proposition, recall that the family of smooth metrics $g_t$ converges pointwise to $g_\ep$. It is then well known (see e.g.~\cite{BU83}) that
the eigenvalue $\la_{k,t}$ converges to $\la_{k,\ep}$ and that there is a basis of eigenfunctions $u_{k,t}$ converging in $H^1(M)$ to a basis of eigenfunctions $u_{k,\ep}$ as
$t\to\infty$. Combining this result with equations~\eqref{conveig} and~\eqref{conveigenf}, the proposition follows.

\section{Proof of Proposition~\ref{evalcol2}}\label{S.lemma}
In order to construct the conformal factor $h_\ep$, we introduce an
auxiliary smooth function $h_s$ depending on two real parameters $s\equiv (s_1,s_2)\in\RR^2$ that has the form
$$
h_s=e^{2s_1\psi_1+2s_2\psi_2}\,,
$$
where $\psi_1$ and $\psi_2$ are smooth functions with zero normal derivative on $\partial\Om$ that will be fixed later. We denote by $\la_{k,\ep,s}$ the eigenvalues of the smooth metric 
$$
g_{\ep,s}:=h_sg_{T}\,,
$$
where $T:=\ep^{-1}$, and by $u_{k,\ep,s}$ an orthonormal basis of the corresponding eigenfunctions. By construction we have that $g_{\ep,s}$ is analytic in the parameters $\ep>0$ and $s\in\RR^2$, and that $g_{\ep,0}=g_{T}$. We also denote by $\mu_{k,s}$ the Neumann eigenvalues of the Laplacian associated with the smooth metric
$$
g_{0,s}:=h_sg_0
$$
in the domain $\Om$, and by $v_{k,s}$ an orthonormal basis of the corresponding eigenfunctions. By the Kato--Rellich theorem, since the Laplacian for the metric $g_{\ep,s}$ is analytic in the parameters $s\in\RR^2$ and $\ep>0$, it is well known that we can take the families of eigenvalues $\la_{k,\ep,s}$ and eigenfunctions $u_{k,\ep,s}$ to vary analytically in $s$ and $\ep$ (this result is independent of the multiplicities of the eigenvalues). The same theorem also implies that the Neumann eigenvalues $\mu_{k,s}$ and eigenfunctions $v_{k,s}$ are analytic in $s$, and in particular
\begin{equation}\label{eqns1}
\lim_{s\to 0}|\mu_{k,s}-\mu_k|=0\,,\qquad \lim_{s\to 0}\|v_{k,s}-v_k\|_{H^1(\Om)}=0\,.
\end{equation}

To construct the desired function $h_\ep$ we will apply a version of the implicit function theorem to the map $\La:\RR^+\times\RR^2\rightarrow\RR^2$ given by
\begin{equation*}
\La(\ep,s):=\begin{cases}
(\la_{1,\ep,s}-\mu_{\DD^2},\la_{2,\ep,s}-\mu_{\DD^2})& \text{if }\ep>0\,,\\
(\mu_{1,s}-\mu_{\DD^2},\mu_{2,s}-\mu_{\DD^2})& \text{if }\ep=0\,.
\end{cases}
\end{equation*}
From the analytic dependence of the eigenvalues with $\ep$ and $s$, we
have that $\La(\ep,s)$ depends analytically on the variable~$\ep\in
(0,\infty)$ for any fixed $s\in\RR^2$ and also on the variable
$s\in\RR^2$ for any fixed $\ep\in[0,\infty)$. In particular, $\La$ is differentiable in its second variable when the first is fixed. Moreover, $\La$ is
continuous in both variables at the point $(0,s_0)$ for any
$s_0\in\RR^2$ and $\La(0,0)=(0,0)$. This is proved by proceeding exactly
as in the proof of Proposition~\ref{lamconv}, cf.~Section~\ref{S:prop}, to show that
$$
\lim_{\ep\searrow0}|\la_{k,\ep,s}-\mu_{k,s}|=0\,,
$$
and
\begin{equation}\label{eqns2}
\lim_{\ep\searrow0}\|u_{k,\ep,s}-v_{k,s}\|_{H^1(\Om)}=0\,,
\end{equation}
for any fixed $s\in\RR^2$ (in the proof one only has to change the reference metric $g_0$ by the metric $g_{0,s}$). Since this property does not depend on the particular way that the point $(\ep,s)$ tends to $(0,s_0)$, we conclude that the map $\La$ is continuous on its domain. 

With the aim of using an implicit function theorem, we now need to evaluate the derivative of $\La$ with respect to~$s$ at the point $(0,0)$. To do this we first calculate the derivative of the Laplacian $\Delta_s$, computed with the metric $g_{0,s}$, at the point $s=0$, acting on the Neumann eigenfunction $v_k$. Noticing that $\Delta_sv_k$ reads as
\begin{align*}
\Delta_sv_k=\frac12h_s^{-2}g_0(\nabla h_s,\nabla v_k)-\mu_kh_s^{-1}v_k\,,
\end{align*}
where $\nabla$ is computed with the reference metric $g_0$, a straightforward computation taking derivatives of this expression with respect to the parameter $s_i$, $i\in\{1,2\}$, and using the form of the conformal factor $h_s$, shows that
\begin{equation}\label{eqdelts}
D_{s_i}\Delta_s(0)v_k=g_0(\nabla \psi_i,\nabla v_k)+2\mu_k\psi_iv_k\,.
\end{equation}

Now we can take derivatives of the eigenvalue equation $\De_s v_{k,s}=-\mu_{k,s}v_{k,s}$ with respect to $s_i$ (recall that the families of eigenfunctions $v_{k,s}$ and $\mu_{k,s}$ were chosen to be analytic with respect to $s$), and evaluate at $s=0$. If we multiply the resulting expression by $v_k$ and integrate over the domain $\Om$ (with respect to the measure induced by $g_0$, which we omit for the sake of simplicity) we easily get
\begin{equation}\label{ders}
D_{s_i}\mu_{k,s}(0)=-\int_{\Om}v_kD_{s_i}\Delta_s(0)v_k-\int_{\Om}v_k\Delta(D_{s_i}v_{k,s}(0))-\mu_k\int_\Om v_kD_{s_i}v_{k,s}(0)\,.
\end{equation}
To obtain this formula we have used that $\int_{\Om}v_k^2=1$. 

Noticing that the unit vector field $\nu_s$ orthogonal to the boundary $\partial\Om$ with respect to the metric $g_{0,s}$ is given by
$$
\nu_s=h_{s}^{-\frac12}\nu\,,
$$
with $\nu$ the unit normal field of $\partial\Om$ computed with $g_0$, we infer that the Neumann eigenfunction $v_{k,s}$ also satisfies the Neumann boundary condition with respect to the normal field $\nu$, that is, $\partial_\nu v_{k,s}=0$. In particular, taking derivatives with respect to $s_i$ in this expression we get $\partial_\nu(D_{s_i}v_{k,s}(0))=0$. 

Accordingly, the $D_{s_i}v_{k,s}(0)$ terms in Eq.~\eqref{ders} cancel after integration by parts, using that $\Delta v_k=-\mu_kv_k$. Finally, putting together Eqs.~~\eqref{eqdelts} and~\eqref{ders} we obtain a Hadamard-type formula for the variation of the Neumann eigenvalue $\mu_{k,s}$:
\begin{align*}
D_{s_i}\mu_{k,s}(0)&=-\int_{\Om}v_kg_0(\nabla \psi_i,\nabla v_k)-2\mu_k\int_\Om\psi_iv_k^2\\&=\frac{1}{2}\int_{\Om}v_k^2(\Delta \psi_i-4\mu_k\psi_i)\,.
\end{align*}
Here we have integrated by parts and used that $\partial_\nu\psi_i=0$, by assumption, in order to pass to the second equality.

This equation and the fact that $\mu_1=\mu_2=\mu_{\DD^2}$ imply that the expression for the derivative $D_s\La(0,0)$ is:
$$
D_s\La(0,0)=\frac12\left(\begin{array}{ll}\int_{\Om}v_1^2(\De \psi_1-4\mu_{\DD^2}\psi_1) & \int_{\Om}v_1^2(\De \psi_2-4\mu_{\DD^2}\psi_2)\\ \int_{\Om}v_2^2(\De \psi_1-4\mu_{\DD^2}\psi_1) & \int_{\Om}v_2^2(\De \psi_2-4\mu_{\DD^2}\psi_2)\end{array}\right)\,.
$$

In order to apply an implicit function theorem to the map $\La$, in the following lemma we compute the rank of the matrix $D_s\La(0,0)$ for a suitable choice of the functions $\psi_1,\psi_2$. The lemma is stated in terms of a function $\hat v_{k_0}$ that is any smooth extension to the whole manifold $M$ of the Neumann eigenfunction $v_{k_0}$ of the domain~$\Om$.
\begin{lemma}\label{deff}
There exists a positive integer $k_0$ such that 
$$
\mathrm{rank}\,[D_s\La(0,0)]=2
$$
with the choice $\psi_1=1$ and $\psi_2=\hat v_{k_0}$. 
\end{lemma}
\begin{proof}
First notice that, with this choice, the functions $\psi_1$ and $\psi_2$ satisfy the Neumann condition $\partial_\nu\psi_i=0$ on $\partial\Om$. Since $\int_\Om v_k^2=1$ by assumption (we are considering an orthonormal basis), the derivative $D_s\La(0,0)$ reads as
$$
D_s\La(0,0)=-\frac12\left(\begin{array}{ll}4\mu_{\DD^2} & (\mu_{k_0}+4\mu_{\DD^2})\int_{\Om}v_1^2v_{k_0}\\ 4\mu_{\DD^2} & (\mu_{k_0}+4\mu_{\DD^2})\int_{\Om}v_2^2v_{k_0}\end{array}\right)\,.
$$
The problem then reduces to finding a Neumann eigenfunction $v_{k_0}$ such that
$$
\int_{\Om}(v_1^2-v_2^2)v_{k_0}\neq 0\,.
$$
It is obvious that such eigenfunction exists because the set of Neumann eigenfunctions $\{v_k\}_{k=0}^\infty$ spans $L^2(\Om)$ and the quantity $v_1^2-v_2^2$ is not identically zero due to the fact that $v_1$ and $v_2$ are not proportional. This proves the lemma.
\end{proof}

We can now use the implicit function theorem~\cite[Theorem B]{Halkin} to obtain a constant $\ep_0>0$, an open neighborhood $U$ of $0\in\RR^2$ and a function $R:[0,\ep_0)\to U$ such that:
\begin{itemize}
	\item $R(0)=0$,
	\item $\La(\ep,R(\ep))=\La(0,0)=0$ for all $\ep\in [0,\ep_0)$, and
	\item $R$ is continuous at $\ep=0$, so in particular $\lim_{\ep\searrow0}R(\ep)=0$.
\end{itemize}

We can define the desired conformal factor $h_\ep$ as 
$$h_{\ep}:=h_{R(\ep)}\,,$$ 
so that the continuity of $R$ implies that $\lim_{\ep\searrow0}\|h_{\ep}-1\|_{C^l(M)}=0$ for any integer~$l$. With this definition of $h_\ep$ we have that the metric $\overline g_\ep$ and the eigenvalues $\overline \la_{k,\ep}$ in the statement of the proposition are given by
\begin{align*}
\overline g_\ep =g_{\ep,R(\ep)}\,,\qquad \overline \la_{k,\ep}=\la_{k,\ep,R(\ep)}\,.
\end{align*}

Accordingly, since $\La(\ep,R(\ep))=0$ for $\ep\in[0,\ep_0)$, we conclude that $$\overline\la_{1,\ep}=\overline\la_{2,\ep}=\mu_{\DD^2}$$
provided that $\ep>0$ is small enough. The multiplicity of this eigenvalue is exactly~$2$ because $\lim_{\ep\searrow 0}|\overline\la_{3,\ep}-\mu_3|=0$ and $\mu_3>\mu_{\DD^2}$.

Concerning the eigenfunctions $\overline u_{k,\ep}$ associated to the eigenvalues $\overline \la_{k,\ep}$, we observe that
$$
\overline u_{k,\ep}=u_{k,\ep,R(\ep)}\,,
$$
and therefore Eqs.~\eqref{eqns1} and~\eqref{eqns2} imply
$$
\lim_{\ep\searrow 0}\|\overline u_{1,\ep}-v_{1}\|_{H^1(\Om)}=\lim_{\ep\searrow 0}\|\overline u_{2,\ep}-v_{2}\|_{H^1(\Om)}=0\,,
$$
thus completing the proof of the proposition.

\section{Higher dimensional analogs of the main theorem}
\label{S.higher}

In Theorem~\ref{T.main} we have shown the existence of a Riemannian metric for any closed $3$-manifold whose first nontrivial
eigenvalue has multiplicity $2$ (i.e., $\la_1=\la_2\neq\la_3$) and a
prescribed component of codimension $2$ (that is, a knot) in the nodal
set $u^{-1}(0)$ of the complex-valued eigenfunction $u=u_1+i
u_2$. In this section we shall show that essentially the same argument
enables us to construct a metric on any closed $d$-dimensional
manifold, whose first nontrivial eigenvalue has multiplicity $m\geq 2$
and a prescribed component of codimension $m$ in the set 
\[
u_1^{-1}(0)\cap\cdots\cap u_m^{-1}(0)\,.
\]

However, a technical condition makes the statement
considerably more involved in the general case. This condition is
associated with the crucial requirement that the nodal set be structurally
stable. In the situation covered by the main theorem, the structural
stability follows from the important equation~\eqref{rank}, which
plays a crucial role in the proof (see Remark~\ref{R:1}). The higher dimensional analog of that
relation is the requirement that the intersection of the nodal sets $u_1^{-1}(0),\ldots,u_m^{-1}(0)$ is transverse on $\Sigma\subset\cap_{k=1}^m u_k^{-1}(0)$, that is,
\begin{equation}\label{trans}
\text{rank}\Big(\nabla u_1(x),\ldots,\nabla u_m(x)\Big)=m
\end{equation}
for all $x\in \Si$.

We shall next state the higher dimensional analog of
Theorem~\ref{T.main}. For this, let $\Si$ be a compact embedded submanifold of a
$d$-dimensional manifold $M$. Throughout we will assume that $\Si$ has
codimension $m\geq 2$ and is connected and boundaryless. Our result
shows that $\Si$ is a connected component of the transverse
intersection of $m$ independent eigenfunctions corresponding to the
same eigenvalue $\la_1$ of the Laplacian on $M$ for some metric. For
the transversality condition~\eqref{trans} to hold, a topological
obstruction on $\Si$ is that its normal bundle must be
trivial. Geometrically, this is equivalent to the assertion that a
small tubular neighborhood of the submanifold~$\Si$ must be
diffeomorphic to $\Si\times\DD^m$, where $\DD^m$ denotes the unit
$m$-dimensional disk. Since a knot always has trivial normal
bundle~\cite{Ma59}, Theorem~\ref{T.main} corresponds exactly to the
case $d=3$ and $m=2$. The case $m=1$, i.e. codimension one
hypersurfaces embedded in $M$, was covered in~\cite{EPpre}.

\begin{theorem}\label{T.high}
Let $\Si$ be a codimension $m$ compact
submanifold of $M$ with trivial normal bundle, $m\geq 2$. Then there
exists a Riemannian metric $g$ on $M$ such that its first nontrivial
eigenvalue has multiplicity $m$ (i.e., $\la_1=\cdots=\la_m\neq \la_{m+1}$) and $\Sigma$ is a connected component of the transverse intersection of the nodal sets $u_1^{-1}(0)\cap\cdots\cap u_m^{-1}(0)$. Here $\{u_1,\cdots,u_m\}$ is any basis of the eigenspace corresponding to $\la_1$. 
\end{theorem}

The proof of this result goes almost exactly as the proof of Theorem~\ref{T.main}. First, we construct a reference metric $g_0$ on $M$ such that 
$$
g_0|_{\Om}=g_\Si+|dz|^2\,,
$$
where $\Om$ is a tubular neighborhood of $\Si$, which can be identified with $\Si\times \DD^m$ because $\Si$ has trivial normal bundle, and $z\in \DD^m$ are coordinates parametrizing the disk $\DD^m$. We take the metric $g_\Si$ on $\Si$ such that the first nontrivial eigenvalue of the Laplacian on $\Si$ with $g_\Sigma$ is bigger than the first Neumann eigenvalue $\mu_{\DD^m}$ of $\DD^m$. Therefore, the Neumann Laplacian in $\Om$ with the metric $g_0$ has eigenvalues $\mu_0=0$,
\[
\mu_1=\cdots=\mu_m=\mu_{\DD^m},
\]
and $\mu_k>\mu_{\DD^m}$ for $k\geq m+1$. The eigenspace of $\mu_1$ is then $m$-dimensional, with a particular orthogonal basis being
\[
v_k=r^{1-\frac{m}{2}}J_{\frac{m}{2}}\left(\sqrt{\mu_{\DD^m}}\,r\right)Y_{1,k}\left(\theta\right) =C_mz_k +O(|z|^2)\,,
\]
where $J_i$ denotes the $i$\textsuperscript{th} order Bessel function and $Y_{1,k}$, $1\leq k\leq m$, are the first order spherical harmonics on $\SS^{m-1}$. The explicit expression of the constant $C_m$ is not relevant for our purposes. From this formula for the eigenfunctions it is easy to check that the intersection
$$
v_1^{-1}(0)\cap\cdots\cap v_m^{-1}(0)
$$
is transverse and coincides with the submanifold $\Si$. The same holds for any other basis of the eigenspace corresponding to $\mu_1$. 

Step~2 is the same as in the proof of Theorem~\ref{T.main}, so that we can construct a $1$-parameter family of smooth metrics $g_t=f_tg_0$ such that the corresponding eigenvalues $\la_{k,t}$ and eigenfunctions $v_{k,t}$ behave when $t\to \infty$ as in Proposition~\ref{lamconv}. In Step~3 we can prove a result analogous to Proposition~\ref{evalcol2}. The only difficulty is in choosing the functions $\psi_i$ (see Section~\ref{S.lemma}) for the conformal factor
$$h_s=e^{2s_1\psi_1+\cdots +2s_m\psi_m}\,,$$ 
where $s\equiv (s_1,\cdots,s_m)\in\RR^m$. The conditions that these functions must satisfy are that  their normal derivatives on $\partial\Om$ must be zero, and the differential $D_s\La(0,0)$ has maximal rank, where the map $\La:\RR^+\times\RR^m\to\RR^m$ is defined as:
\begin{equation*}
\La(\ep,s):=\begin{cases}
(\la_{1,\ep,s}-\mu_{\DD^m},\cdots, \la_{m,\ep,s}-\mu_{\DD^m})& \text{if }\ep>0\,,\\
(\mu_{1,s}-\mu_{\DD^m},\cdots,\mu_{m,s}-\mu_{\DD^m})& \text{if }\ep=0\,.
\end{cases}
\end{equation*}
The quantities $\la_{k,\ep,s}$ and $\mu_{k,s}$ are the eigenvalues corresponding to the metrics $g_{\ep,s}:=h_s g_T$, $T=\ep^{-1}$, and $g_{0,s}:=h_sg_0$, respectively, with their associated orthonormal eigenfunctions $u_{k,\ep,s}$ and $v_{k,s}$. By Kato--Rellich we can assume that all these quantities are analytic in the variables $\ep>0$ and $s\in\RR^m$.  

The derivative $D_s\La(0,0)$ can be computed using a Hadamard-type formula as in Section~\ref{S.lemma}, thus yielding:
$$
(D_s\La(0,0))_{ij}=\frac12\int_\Om v_i^2((d-2)\De \psi_j-4\mu_{\DD^m}\psi_j)
$$
for $1\leq i,j\leq m$. To compute the rank of this matrix, we can prove an equivalent lemma to Lemma~\ref{deff} by noting that the transversality of the nodal sets of $v_k$, $1\leq k\leq m$, implies that the functions $v_k^2$ are linearly independent and hence their span $V$ is $m$-dimensional. This allows us to choose $m$ linearly independent Neumann eigenfunctions $\{v_{k_1},\ldots,v_{k_{m}}\}$ so that no function in their span is orthogonal to the subspace $V$ (these eigenfunctions exist because the set of Neumann eigenfunctions is a basis of $L^2(\Om)$).

Therefore, if we set
\begin{equation*}
\psi_i=\hat v_{k_{i}}
\end{equation*}
for $i\in\{1,\ldots, m\}$, where $\hat v_k$ denotes any smooth extension to the whole manifold $M$ of the Neumann eigenfunction $v_k$, it follows that the vectors
\begin{align*}
w_i:=-\frac12\left(\int_\Om\Big((d-2)\mu_{k_1}+4\mu_{\DD^m}\Big)v_i^2v_{k_1},\cdots,
\int_\Om\Big((d-2)\mu_{k_m}+4\mu_{\DD^m}\Big)v_i^2v_{k_m}\right)
\end{align*}
for $1\leq i\leq m$, are linearly independent, where each vector $w_i$ is the $i^{\mathrm{th}}$ row of the matrix $D_s\La(0,0)$. Indeed, if $w_i$ were linearly dependent, there would exist a nontrivial set of constants $\{c_i\}_{i=1}^m\subset \RR$ such that
\begin{equation*}
\sum_{i=1}^m c_i\int_{\Omega}v_i^2 v_{k_j}=0
\end{equation*}
for all $1\leq j\leq m$. Hence, the function $F:=\sum_{i=1}^m c_iv_i^2$ would be a linear combination of the functions $v_i^2$ that is orthogonal to each $v_{k_j}$. By the construction of the Neumann eigenfunctions $v_{k_j}$, this would imply that $F=0$ and from the linear independence of the set $\{v_i^2\}_{i=1}^m$ we would obtain that $c_i=0$ for every $1\leq i\leq m$, thus proving the claim. Notice that, as in the proof of Proposition~\ref{evalcol2}, we can take $\psi_1=1$.

The rest of the proof goes exactly as in Section~\ref{S.lemma} using the implicit function theorem~\cite[Theorem B]{Halkin} to prove a result analogous to Proposition~\ref{evalcol2}, and Thom's isotopy theorem in Step~4, so the details are omitted.

\begin{remark}
Theorem~\ref{T.high} has the following striking corollary: since any
exotic sphere of dimension~$m$ embeds smoothly in $\SS^{2m}$ with trivial normal bundle~\cite{exotic}, we conclude that there is a metric on $\SS^{2m}$ whose first nontrivial eigenvalue has multiplicity $m$ and the exotic sphere is a component of the transverse intersection of the nodal sets $u_1^{-1}(0)\cap\cdots\cap u_m^{-1}(0)$. Recall that an exotic sphere is a smooth manifold that is homeomorphic but not diffeomorphic to the standard sphere. 
\end{remark}

\section*{Acknowledgments}
The authors are supported by the ERC Starting Grants~633152 (A.E.) and~335079
(D.H.\ and D.P.-S.). This work is supported in part by the
ICMAT--Severo Ochoa grant
SEV-2011-0087.

\bibliographystyle{amsplain}

\end{document}